\documentclass[a4paper]{amsart}

\usepackage{amsthm,verbatim,amsfonts,amscd}
\usepackage{amssymb,amsmath}
\usepackage{pdflscape}
\usepackage{mathrsfs}
\usepackage{tikz}
\usepackage{graphicx}
\usepackage{caption}
\usepackage{subcaption}
\usepackage{afterpage}
\usepackage{hyperref}
\usepackage[T1]{fontenc}

\newcommand\floor[1]{\lfloor#1\rfloor}
\newcommand\ceil[1]{\lceil#1\rceil}
\usetikzlibrary{calc,intersections, arrows, shapes}

\newcommand{\CC}{\mathbb{C}}
\newcommand{\BB}{\mathbb{B}}

\newcommand{\RR}{\mathbb{R}}

\newcommand{\NN}{\mathbb{N}}

\def\psd{\ensuremath{\textup{psd}}}
\def\psdc{\ensuremath{\psd^{\CC}}}
\newcommand{\rank}{\textup{rank}\,}
\newcommand{\rankplus}{\textup{rank}_+\,}
\newcommand{\rankpsd}{\textup{rank}_{\textup{psd}}}

\newcommand{\xc}{\textup{xc}}
\newcommand{\xcs}{\textup{xc}_{\psd}}
\newcommand{\xcb}{\textup{xc}_{\BB}}
\newcommand{\xcsc}{\textup{xc}_{\psd}^{\CC}}

\newcommand{\tr}{\textcolor{red}}
\def\rp{\ensuremath{\rank^{\CC}_{\psd}}}
\def\rankb{\ensuremath{\rank_{\BB}}}
\def\rankh{\ensuremath{\rank^{\textup{hom}}_{\BB}}}
\newcommand*\conj[1]{\overline{#1}}

\newtheorem{theorem}{Theorem}[section]
\newtheorem{conjecture}[theorem]{Conjecture}
\newtheorem{lemma}[theorem]{Lemma}
\newtheorem{corollary}[theorem]{Corollary}
\newtheorem{proposition}[theorem]{Proposition}
\theoremstyle{definition}
\newtheorem{definition}[theorem]{Definition}
\newtheorem{example}[theorem]{Example}

\theoremstyle{remark}
\newtheorem{remark}[theorem]{Remark}

\begin{document}
\title[On ranks of regular polygons]{On ranks of regular polygons}
\author[Goucha]{Ant\'onio Pedro Goucha}
\address{Department of Mathematics,  University of Coimbra, 3001-454 Coimbra, Portugal}{}
\email{apngoucha@student.uc.pt}

\author[Gouveia]{Jo\~ao Gouveia}
\address{CMUC, Department of Mathematics,
  University of Coimbra, 3001-454 Coimbra, Portugal}{}
\email{jgouveia@mat.uc.pt}

\author[Silva]{Pedro M. Silva}
\address{Department of Physics,  University of Coimbra, 3001-454 Coimbra, Portugal}{}
\email{pmsilva@student.fisica.uc.pt}

\thanks{Goucha was partially funded by a PhD scholarship from Funda{\c c}{\~ a}o para a Ci{\^ e}ncia e Tecnologia. Gouveia was partially supported by the Centre for Mathematics of the University of Coimbra -- UID/MAT/00324/2013, funded by the Portuguese
Government through FCT/MEC and co-funded by the European Regional Development Fund through the Partnership Agreement PT2020. Silva was partially supported through a scholarship of the program Novos Talentos em Matem\'atica of the Calouste Gulbenkian Foundation.}

\begin{abstract}
In this paper we study various versions of extension complexity for polygons through the study of factorization ranks of their slack matrices. In particular, we develop a new asymptotic lower bound for their nonnegative rank, shortening the gap between the current bounds, we introduce a new upper bound for their boolean rank, deriving from it some new numerical results, and we study their complex semidefinite rank, uncovering the possibility of non monotonicity of the ranks of regular $n$-gons.
\end{abstract}

\maketitle


\section{Introduction} \label{sec:introduction}

Polytopes play a central role in optimization, among other reasons because they are natural objects to represent a wide variety of optimization problems. Thus, a great interest has developed in studying the existence of efficient representations for them. In rough terms, the difficulty of optimizing over a polytope using a standard LP algorithm grows polynomially with its number of facets or vertices. To circumvent this fact, one can try to write a polytope as a linear image of a simpler, albeit higher dimensional, object. Depending on which objects we consider we get different measures of complexity.

Given a polytope $P$, a \emph{linear extension} of $P$ is a polytope $Q$ such that there exists a linear map $\pi$ with $\pi(Q)=P$. The \emph{size} of such an extension is defined to be the number of facets of $Q$ and we say that the \emph{(linear) extension complexity} of $P$, $\xc(P)$, is the smallest size of any linear extension of $P$.

Similarly, for a given polytope $P$, a \emph{semidefinite extension} of $P$ is a spectrahedron $S$ for which, again, there is a linear map $\pi$ with $\pi(S)=P$. Recall that a spectrahedron is simply a set of the form
$$S=\left\{ x \in \RR^k \textup{ s.t. } A_0+\sum_{i=1}^k x_i A_i \succeq 0\right\}$$
where $A_i$, $i=0,\cdots,k$, are real symmetric matrices and $M \succeq 0$ means $M$ is positive semidefinite. The \emph{size} of a semidefinite extension is the dimension of the matrices $A_i$ used in defining it, and the \emph{semidefinite extension complexity} of $P$, $\xcs(P)$ is the smallest possible size of any of its semidefinite extensions.

Finally, an analogous notion is that of a \emph{complex semidefinite extension} of a polytope $P$, which is obtained by using complex matrices in the spectrahedra definition instead of real as before. We now want to write $P$ as $\pi(T)$ for
$$T=\left\{ x \in \RR^k \textup{ s.t. } B_0+\sum_{i=1}^k x_i B_i \succeq 0\right\}$$
where $B_i$, $i=0,\cdots,k$, are hermitian matrices. The \emph{complex semidefinite extension complexity} of $P$, $\xcsc(P)$, is then defined in the analogous way to $\xcs(P)$.

As pointed out by Yannakakis in his seminal work \cite{yannakakis1988expressing}, the study of the extension complexity of a polytope can be done in terms of restricted factorizations of certain matricial representations of it. Given a polytope $P$, two common ways to represent it are listing its vertices, $p_1, \cdots, p_v$, usually said to be a $V$-representation of $P$, or listing linear inequalities corresponding to its facets, $h_1(x) \geq 0, \cdots, h_f(x) \geq 0$, usually said to be an $H$-representation of $P$. These two dual forms of representing a polytope can be combined into a matrix $S_P$, called the \emph{slack matrix} of $P$, defined as $S_P(i,j)=h_i(p_j)$.
The slack matrix of $P$ is defined only up to scaling of rows by positive factors, as the facet inequalities can be scaled. Two important things to keep in mind is that the slack matrix is nonnegative and it always has rank $d+1$, where $d$ is the dimension of $P$. A thorough study of such matrices can be found in \cite{gouveia2013nonnegative}.

To connect $S_P$ back to the extension complexity of $P$ one has to introduce some restricted factorizations and their respective ranks. Given a nonnegative $m \times n$ matrix $M$, a \emph{nonnegative factorization} of $M$ is a
factorization $M=AB^T$ with $A$ and $B$ nonnegative matrices. The \emph{size} of such factorization is the inner dimension of the matrix product, i.e. $k$ if $A$ is $m \times k$, and the \emph{nonnegative rank} of $M$, $\rankplus(M)$, is the smallest size of such a nonnegative factorization. In turn, a \emph{semidefinite factorization} of $M$ of size $k$ is a collection of $k \times k$ real positive semidefinite matrices $A_1,\cdots,A_m$ and $B_1, \cdots, B_n$ such that $M_{i,j}=\langle A_i,B_j \rangle$ for all entries $(i,j)$ of $M$, and a \emph{complex semidefinite factorization} is defined similarly with recourse to $A_i$ and $B_j$ complex semidefinite matrices. Here the inner product considered is the usual trace product $\langle A,B \rangle = \tr(B^*A)$.
As before, the \emph{semidefinite rank} of $M$, $\rankpsd(M)$ (respectively the \emph{complex semidefinite rank} of $M$, $\rp(M)$), is the smallest size of a semidefinite (respectively complex semidefinite) factorization of $M$.
Yannakakis result \cite{yannakakis1988expressing} for the linear extension complexity and its extension to general cones \cite{GPT} connect these notions back to extension complexity in a very elegant way.

\begin{theorem}
For any polytope $P$, $\xc(P)=\rankplus(S_P)$, $\xcs(P)=\rankpsd(S_P)$ and $\xcsc(P)=\rp(S_P)$.
\end{theorem}

This connection between extension complexities and ranks has been recently explored in several groundbreaking results lower bounding the complexity of linear and semidefinite formulations of classic combinatorial problems. See for example \cite{fiorini2015exponential}, \cite{rothvoss2014matching} and \cite{LRS}. For simplicity, we will slightly abuse the definitions and refer to the extension complexities of polygons as ranks of polygons.

There are obvious relations among the factorization ranks defined above. For instance, for any nonnegative matrix $M$, we have $\rank(M) \leq \rankplus(M)$ and also $\rp(M) \leq \rankpsd(M) \leq \rankplus(M)$. Another important relation of the nonnegative rank is with yet another factorization rank, the \emph{boolean rank}. The boolean rank of a matrix $M$ of size $m \times n$, $\rankb(M)$, can be defined as the smallest $k$ for which one can find zero-one matrices $C$ and $D$ of sizes $m \times k$ and $n \times k$, such that $CD^T$ has the same support as $M$ i.e., the same set of nonzero entries. It is easy to see that $\rankb(M) \leq \rankplus(M)$ and this fact has been instrumental in many of the efforts in lower bounding the nonnegative rank of matrices. Incidentally, similarly to the other factorization ranks, one can interpret the boolean rank in terms of extension complexities. If we define the combinatorial extension complexity of a polytope $P$, $\xcb(P)$, as the smallest number of facets of a polytope $Q$ for which the face lattice of $P$ can be embedded in the face lattice of $Q$, then one can see that $\xcb(P)=\rankb(S_P)$ (see Corollary 2.13 of \cite{kaibel2013constructing}).

In this work we will focus on a restricted class of polytopes, polygons, and particularly regular polygons. This is in a sense the first non trivial class of polytopes one can study, and there were several recent important efforts in understanding their extension complexity. For regular $n$-gons there were several upper bounding results like \cite{ben2001polyhedral}, \cite{fiorini2012extended} and \cite{vandaele2015linear}, while in the general polygon case we also have the upper bounds in \cite{shitov2014upper} and \cite{1412.0728}. There is also interesting numerical evidence on the true nonnegative and semidefinite ranks of regular polygons presented in \cite{thesisvand} and \cite{vandaele2016}. Even with all these contributions, much is left to determine about the extension complexity of polygons, and this paper intends to be an addition to this growing set of results.

This paper is organized as follows. In Section \ref{sec:nonnegative} we focus on nonnegative rank. We provide a complete study of a geometric lower bound, providing the first asymptotic lower bound to the nonnegative rank of polygons that beats the trivial $\log_2(n)$ bound, and shortens the gap between upper and lower asymptotic bounds. After that we turn to the Boolean rank in Section \ref{sec:boolean}, defining a new upper bound for it and using it to numerically study $n$-gons for small $n$, in some cases determining their true boolean ranks. Finally, in Section \ref{sec:psd} we study the complex semidefinite rank of polygons, develop some tools to study the minimal cases and use them to uncover the possibility of nonmonotonicity of ranks of regular $n$-gons as $n$ grows.

\section{Nonnegative rank}\label{sec:nonnegative}

In this section we concern ourselves with an asymptotic study of the lower bounds for the nonnegative rank of polygons. The most common lower bounds for the nonnegative rank are combinatorial in nature and in most cases are actually lower bounds for the boolean rank. Examples of that are the trivial $\log_2(n)$ lower bound for the nonnegative rank of an $n$-gon, or its improvement
$$S(n)= \min \left\{k: n\leq \binom{k}{\floor{k/2}}\right\}.$$
When applied to polygons however these combinatorial bounds seem not to be very useful. $S(n)$, for example, can be shown to be asymptotically equivalent to $\log_2(n)$, i.e. $\lim S(n)/\log_2(n)=1$, while other common bounds turn out to be void of information, like the fooling number lower bound, that always equals $4$ for any polygon.

In this section we shorten the gap between this $\log_2(n)$ asymptotic lower bound to the $2\log_2(n)$ asymptotic upper bound proven in \cite{ben2001polyhedral}, conjectured in \cite{vandaele2015linear} to be the true value, by improving the lower bound to approximately $1.44 \log_2(n)$.

\subsection{Geometric lower bound}

The lack of effectiveness of the usual combinatorial bounds makes it necessary to introduce some geometric reasoning in order to accomplish some nontrivial result.
To this end we will resort to McMullen's Upper Bound Theorem.
Recall that the cyclic polytope $\mathcal{C}(n,d)$ is the convex hull of $n$ distinct points on the moment curve $\{(\tau, \tau^2,..., \tau^d)\, : \, \tau \in \RR\}$.

\begin{theorem}[Upper Bound Theorem \cite{MTK:6718596}]
 For fixed $n$ and $d$ the maximum number of $i$-faces for a $d$-polytope with $n$ vertices is attained by $\mathcal{C}(n,d)$, for $i=0,\cdots,d$. By duality,
 the maximum number of $i$-faces for a $d$-polytope with $n$ facets is attained by $\mathcal{C}^*(n,d)$, the dual of the cyclic polytope.
\end{theorem}

Let $Q$ be a polytope with $m$ facets that is a linear extension of a $d$-polytope $P$. On the one hand, being a linear extension implies $f_0(Q) \geq f_0(P)$, on the other, the Upper Bound Theorem implies
$$f_0(Q)\leq f_0(\mathcal{C}^*(m,\dim(Q)))=f_{\dim(Q)-1}(\mathcal{C}(m,\dim(Q)))\leq \max_{2 \leq d \leq m-1} f_{d-1}(\mathcal{C}(m,d)).$$
These two facts combined allow us to lower bound the number of facets of any linear extension of $P$.
\begin{definition}
$T(v):= \min_{k}\{v \leq \max_{2 \leq d \leq k-1} f_{d-1}(\mathcal{C}(k,d))\}$.
\end{definition}
Keeping in mind the inequality above, we get $T(f_0(Q)) \leq m$. Since $f_0(P) \leq  f_0(Q)$, it follows that $T(f_0(P)) \leq T(f_0(Q)) \leq m$. In particular, $T(f_0(P)) \leq \rank_+(P)$. Note that this lower bound relates closely to the one introduced in \cite{vandaele2015linear} and is actually the same in the case of polygons.

Finally, we just note that as expected this lower bound can be strict. For $P$ a $9$-gon we have $T(9)=6$ and the bound gives us $\rank_+(P) \geq 6$. However, any polytope with $6$ facets and $9$ vertices is combinatorially equivalent to a product of triangles and in \cite{RORIG201279} it is proven that the projection of such a polytope to the plane has at most $8$ vertices, hence $\rank_+(P) \geq 7$.

\subsection{Asymptotic study}

The first step in order to study $T(n)$ is to find a simplified expression for it that avoids having to take the maximum.
The expression for $f_{d-1}(\mathcal{C}(k,d))$ can be found, for example, in \cite{grunbaum1967convex}.

\[
    \ f_{d-1}(\mathcal{C}(k,d))=
\begin{cases}
    \frac{k}{k-n}\binom{k-n}{n},& \text{if } d = 2n;\\
    2\binom{k-n-1}{n},              & \text{if } d = 2n+1.
\end{cases}
\]

We want to compute $\max_{2 \leq d \leq k-1} f_{d-1}(\mathcal{C}(k,d))$. In order to do that, we will separate the odd and even cases. For the odd case we can use
the following result.

\begin{proposition}[Tanny and Zuker \cite{tanny1974unimodal}]
 For fixed $n$, let $r_n$ be the smallest integer for which $\binom{n-r}{r}$ is maximal. Then, $r_n = \floor{\frac{1}{2}n(1-\frac{\sqrt{5}}{5})}$ or $r_n = \floor{\frac{1}{2}n(1-\frac{\sqrt{5}}{5}) + 1}$.
\end{proposition}

Applying this to the general term of the odd subsequence, $2\binom{k-1-n}{n}$, we get that its maximizer is
 $n = \floor{\frac{1}{2}(k-1)(1-\frac{\sqrt{5}}{5})}$ or $n = \floor{\frac{1}{2}(k-1)(1-\frac{\sqrt{5}}{5}) + 1}$.
Denote it by $m_1(k)$. We proceed to study the even subsequence.

\begin{lemma}
The maximum of  $\frac{k}{k-n}\binom{k-n}{n}$, $n \in \mathbb{N}$ and $1\leq n\leq \frac{k-1}{2}$, is attained at $m_2(k)= \ceil{\frac{5k-4-\sqrt{5k^2-4}}{10}}$.
\end{lemma}

\begin{proof}

Let $t_n = \frac{k}{k-n}\binom{k-n}{n}$. We will study the ratio between consecutive terms in order to study its monotony
$$
q(n)=\frac{t_{n+1}}{t_n}=\frac{\frac{k}{k-(n+1)}\binom{k-(n+1)}{n+1}}{\frac{k}{k-n}\binom{k-n}{n}} = \frac{(k-2n)(k-2n-1)}{(n+1)(k-n-1)}.
$$
The ratio $q(n)$ is defined for $1\leq n\leq \frac{k-3}{2}$ and we need to determine when is it less than or equal to $1$. We will study its canonical real extension to the interval $[1, \frac{k-3}{2}]$, since it contains all our interest points. Since in this interval the denominator never vanishes, by expanding
$$q(x)=\frac{(k-2x)(k-2x-1)}{(x+1)(k-x-1)} \leq 1$$
one can easily check that it is equivalent to
$$5x^2+(4-5k)x+(k-1)^2 \leq 0.$$
The polynomial $P(x)=5x^2+(4-5k)x+(k-1)^2$ has roots $r_1 = \frac{5k - 4 - \sqrt{5k^2 - 4}}{10}$ and $r_2 = \frac{5k - 4 + \sqrt{5k^2 - 4}}{10}$ and is nonpositive precisely in $[r_1, r_2]$. Since $r_2 > \frac{k-3}{2}$, for any integer $n$ in the domain interval $P(n)>0$ if and only if $n<r_1$, hence $t_{n+1} > t_n$ if and only if $n<r_1$. We conclude that the maximizer of $t_n$ coincides with the least integer belonging to $[r_1; r_2]$, that is to say, $\ceil{r_1}$.
\end{proof}

We now just have to compare the values of the maxima of both subsequences.

\begin{corollary}
If $k\neq5$, $\max_{2 \leq d \leq k-1} f_{d-1}(\mathcal{C}(k,d)) = \frac{k}{k-m_2(k)}\binom{k-m_2(k)}{m_2(k)}$.
\end{corollary}

\begin{proof}
We start by relating the expressions of both subsequences:
$$
\frac{k}{k-n}\binom{k-n}{n} = \frac{k(k-n-1)!}{n!(k-2n)!} = \frac{k}{k-2n}\binom{k-1-n}{n}.
$$
If $n \geq \frac{k}{4}$, we have
$$
\frac{k}{k-n}\binom{k-n}{n} = \frac{k}{k-2n}\binom{k-1-n}{n} \geq 2\binom{k-1-n}{n}.
$$
Since $m_2(k)$ maximizes the original expression, when $m_1(k) \geq \frac{k}{4}$ we attain from the inequality above
\begin{align*}
\frac{k}{k-m_2(k)}\binom{k-m_2(k)}{m_2(k)} \geq \frac{k}{k-m_1(k)}\binom{k-m_1(k)}{m_1(k)} \geq 2\binom{k-1-m_1(k)}{m_1(k)}.
\end{align*}

To guarantee $m_1(k) \geq \frac{k}{4}$, it is enough that $\frac{1}{2}(k-1)(1-\frac{\sqrt{5}}{5}) - 1 \geq \frac{k}{4}$, which happens for $k\geq 49$. Computing numerically $\max_{2 \leq d \leq k-1} f_{d-1}(\mathcal{C}(d,k))$ for $ k \leq 48$, we can see that it is always attained in $m_2(k)$, except for $k=5$ which is attained in $m_1(k)$.
\end{proof}

For asymptotic reasonings one can then consider $$T(n)= \min_{k}\left\{n \leq \frac{k}{k-m_2(k)}\binom{k-m_2(k)}{m_2(k)}\right\}.$$
Note that this expression now presents many similarities to that of $S(n)$. The next step is to reduce the asymptotic study of $T(n)$ to that of a subsequence of $T(n)$ with an easier expression to work with, so that we can get rid of the minimization in the definition.

\begin{lemma}\label{lem:Tsubsucessao} Let $s_n=\frac{n}{n-m_2(n)}\binom{n-m_2(n)}{m_2(n)}$. Then,
$\lim_{n} \frac{T(n)}{\log_2(n)}$ exists if and only if $\lim_{n} \frac{T(s_n)}{\log_2(s_n)}$ does, in which case they are equal.
\end{lemma}

\begin{proof}
For the direct implication just note that $s_n$ is, by construction, the maximum number of facets of a polytope with $n$ vertices and hence easily seen to be increasing. Therefore $\frac{T(s_n)}{\log_2(s_n)}$ is a subsequence of $\frac{T(n)}{\log_2(n)}$ and converges if this one does.

For the reverse implication note that $T(s_n)=n$, since $s_n$ is increasing. Then, for $s_k < n \leq s_{k+1}$, we have
\begin{align*}
\frac{T(n)}{\log_2(n)} = \frac{k+1}{\log_2(n)} = \frac{T(s_{k+1})}{\log_2(n)} \geq \frac{T(s_{k+1})}{\log_2(s_{k+1})}
\end{align*}
and
\begin{align*}
\frac{T(n)}{\log_2(n)} =  \frac{T(s_{k})}{\log_2(n)} + \frac{1}{\log_2(n)} \leq  \frac{T(s_{k})}{\log_2(s_k)} + \frac{1}{\log_2(s_k)}.
\end{align*}
If the limit of  $\frac{T(s_n)}{\log_2(s_n)}$ exists, taking limits we conclude that the limit of $\frac{T(n)}{\log_2(n)}$ is the same.
\end{proof}

This result reduces the asymptotic study of $T(n)$ to that of $\frac{n}{\log_2(s_n)}$. To do that there is a classic tool that we can use, Stirling's approximation, that states
$$\ln(n!) = n \ln n - n + \mathcal{O}(\ln n).$$
Applying it to binomial coefficients we get the well-known approximation
$$\ln(\binom{n}{m}) =  n\ln n + \mathcal{O}(\ln n) - m\ln m + \mathcal{O}(\ln m)- (n-m)\ln(n-m) - \mathcal{O}(\ln(n-m)).$$
With this result we can now prove the intended result.

\begin{corollary}
The sequence $T(n)$ is asymptotically equivalent to $\log_\phi (n)$, where $\phi$ is the golden ratio.
\end{corollary}
\begin{proof}
By Lemma \ref{lem:Tsubsucessao}, studying $\lim_{n} \frac{\log_2(n)}{T(n)}$, is the same as studying $\lim_{n} \frac{\log_2(s_n)}{n}$. Furthermore,
$$\log_2\left(s_n\right) = \frac{1}{\ln 2} \left(\ln \left( \frac{n}{n-m_2(n)}\right)+\ln \left(\binom{n-m_2(n)}{m_2(n)}\right) \right),$$
thus
\begin{align*}
\lim_{n} \frac{\log_2(s_n)}{n} =
\frac{1}{\ln 2} \lim_{n} \frac{\ln(\binom{n-m_2(n)}{m_2(n)})}{n},
\end{align*}
since $\frac{\ln(\frac{n}{n-m_2(n)})}{n} < \frac{\ln n}{n}$.

Using the approximation for $\ln(\binom{n}{m})$ in the expression for $\ln(\binom{n-m_2(n)}{m_2(n)})$, and noting that $\frac{\mathcal{O}(\ln (n-m_2(n)))}{n}$,
$\frac{\mathcal{O}(\ln(m_2(n)))}{n}$ and $\frac{\mathcal{O}(\ln(n-2m_2(n)))}{n}$ all go to zero as $n$ goes to infinity, one gets that $\lim_{n} \frac{\log_2(s_n)}{n}$ is the sum of
$$\frac{1}{\ln 2}\lim_{n} \frac{(n-m_2(n))\ln(n-m_2(n))}{n}$$
and
$$\frac{1}{\ln 2}\lim_{n} \frac{-m_2(n)\ln(m_2(n))-(n-2m_2(n))\ln(n-2m_2(n))}{n}.$$

To compute these, recall that $m_2(n)= \ceil{\frac{5n-4-\sqrt{5n^2-4}}{10}}$ thus $m_2(n) \sim r(n)$,with $r(n)=\frac{1}{2}n(1-\frac{\sqrt{5}}{5})$. Likewise, $n-m_2(n) \sim n-r(n)$ and $n-2m_2(n)\sim n-2r(n)$. Therefore the limit we are interested in is
\begin{align*}
\frac{1}{\ln 2} \lim_{n} \frac{(n-r(n))\ln(n-r(n))-r(n)\ln(r(n))-(n-2r(n))\ln(n-2r(n))}{n}.
\end{align*}
Replacing $r(n)$ by its expression one gets
\begin{align*}
\frac{1}{\ln 2} \lim_{n} \frac{\frac{n}{2}(1+\frac{\sqrt{5}}{5})\ln(\frac{1}{2}+\frac{\sqrt{5}}{10})-\frac{n}{2}(1-\frac{\sqrt{5}}{5})\ln(\frac{1}{2}-\frac{\sqrt{5}}{10})-n\frac{\sqrt{5}}{5}\ln(\frac{\sqrt{5}}{5})}{n}.
\end{align*}
Simplifying we can see that this is precisely $\log_2(\phi)$, where $\phi$ is the golden ratio. Therefore
$$\lim_{n} \frac{T(n)}{\log_2(n)} = \frac{1}{\log_2(\phi)} = \log_\phi(2),$$
 and $\lim_{n} \frac{T(n)}{\log_\phi(n)} = 1$, as intended.
\end{proof}

With this result we now know the asymptotic lower bound for the extension complexity of an $n$-gon of $\log_\phi(2) \log_2(n)$, which is approximately $1.44 \log_2(n)$.
This gives some convincing evidence that purely combinatorial tools to lower bound nonnegative ranks can be massively improved even by the simple addition of basic geometric reasoning. Closing the gap to the $2 \log_2(n)$ upper bound will likely need a much more sophisticated reasoning.

\section{Boolean Rank}\label{sec:boolean}

The fact that the combinatorial lower bounds are not effective in lower bounding the nonnegative rank of polygons seems to suggest that the boolean rank is actually asymptotically lower than the nonnegative rank. In fact, there are no reasons to believe that the boolean rank of an $n$-gon is not asymptotically $\log_2(n)$, exactly as its trivial lower bound. There are, however, no effective tools to upper bound the boolean rank, and very little numerical evidence either way. This fact leads us in this section to take inspiration from \cite{barefoot1986biclique} to formally introduce a new upper bound for the boolean rank of a polygon, and to do some numerical experiments that improve and expand those presented in that paper.

\subsection{Homogeneous boolean rank of a polygon}

The paper \cite{barefoot1986biclique} studies the boolean rank of both polygons and a family of related matrices. It derives a few basic properties and carries out some numerical experiments. In order to make part of what is done there a little more concrete, we will introduce a new definition that gives a restricted version of the boolean rank. The main idea is to restrict the rows of the matrices used in the boolean factorization to have supports of fixed cardinality, which allows us to look into a much smaller set of possible factorizations. In order to have as much choice as possible, we will set the supports to have cardinality as close to half the size of the rows as possible.

\begin{definition}\label{def:rankhomogeneous}
A \emph{homogeneous boolean factorization} of size $k$ of the slack matrix $S_n$ of an $n$-gon is a boolean factorization of the type $C_{n \times k}\times D^T_{n \times k}$ where the rows of $C$ have precisely $\floor{\frac{k}{2}}$ ones and those of $D$ precisely $\ceil{\frac{k}{2}}-1$ ones.
The \emph{homogeneous boolean rank} of $S_n$, $\rankh(S_n)$, is the smallest $k$ for which such a factorization exists.
\end{definition}

The first thing to notice is that for any $n$-gon its homogeneous boolean rank is finite. To see this just note that one can pad the trivial factorization $S_n=S_n \times I_n$ to obtain $$S_n = [0_{n \times n-3} S_n] \times [1_{n \times n-3} I_n]^T ,$$
which is a homogeneous boolean factorization of size $2n-3$, hence $\rankh(S_n) \leq 2n-3$. This reasoning depends only on the fact that all rows of $S_n$ have precisely the same number of zeroes, and we could easily extend this homogeneous boolean rank definition to $d$-dimensional simplicial polytopes by demanding $\floor{\frac{k}{2}}-\floor{\frac{d}{2}}+1$ and $\ceil{\frac{k}{2}}+\ceil{\frac{d}{2}}$ ones in the rows of $C$ and $D$ respectively. Our aim in this paper is however limited only to polygons so we will only work with Definition \ref{def:rankhomogeneous}.

The second thing to notice is that trivially $\rankh(S_n) \geq \rankb(S_n)$. However in \cite{barefoot1986biclique} a stronger relationship is suggested.

\begin{conjecture}\label{conj:homog}
For any $n$, $\rankh(S_n) = \rankb(S_n)$.
\end{conjecture}

In this section we will give a graph interpretation of the notion of homogeneous boolean rank, and use it to compute some of its values, attaining a few maximizers to the boolean rank of $n$-gons. Some of these numerical experiments were carried out also in \cite{barefoot1986biclique}, but it is not always clear in that paper what conditions were being assumed. We aim to expand those results and clarify some details while also to making available the code and the factorizations attained so that others can verify the results independently or use the factorizations to test conjectures on asymptotic upper bounds for the boolean rank.

Let $C_{n \times k}\times D^T_{n \times k}$ be a homogeneous boolean rank factorization of $S_n$. We will identify each row of $C$ with its support, which in turn can be identified with a vertex of a Johnson graph.
\begin{definition}
For given $n$ and $k$, The Johnson graph $J(n,k)$, is the graph whose vertices are cardinality $k$ subsets of $\{1,\dots,n\}$, with two vertices being adjacent if they share $k-1$ elements.
\end{definition}

\begin{figure}[t]
  \centering
    \centerline{\includegraphics[width=0.4\textwidth]{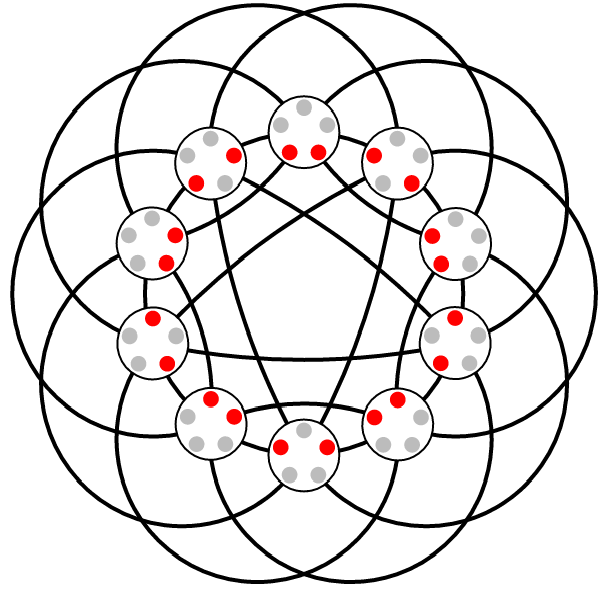}}
	\caption{Johnson graph $J(5,2)$.}
   \label{fig:johnson}
\end{figure}

In Figure \ref{fig:johnson} we can see a representation of $J(5,2)$.
Let $C_i$ and $D_j$ stand for the sets corresponding to rows $i$ and $j$ of matrices $C$ and $D$ respectively. Since $S_n$ has zeros in positions $(l,l)$ and $(l+1,l)$, $D_l$ must be disjoint of $C_l$ and $C_{l+1}$. Since the complement of $D_l$ has $\floor{\frac{k}{2}}+ 1$ elements and both $C_l$ and $C_{l+1}$ are contained there and have $\floor{\frac{k}{2}}$ elements, the fact that they are distinct implies that they must have $\floor{\frac{k}{2}}-1$ elements in common, hence they correspond to vertices of $J(k,\floor{\frac{k}{2}})$ connected by an edge. Therefore $(C_1,C_2,\cdots,C_n,C_1)$ must be a cycle of this Johnson graph.

To get from cycles to factorizations we need extra conditions. Define a coloring on the edges of the Johnson graph by coloring edge $\{S,T\}$ with a color corresponding to the complement of their union, $\overline{S \cup T}$.
If we have a cycle as above, coming from a factorization, the color of edge $\{C_l,C_{l+1}\}$ is precisely $D_l$ which, since all rows of $D$ are distinct, implies that it must have all edges of distinct colors i.e., it must be a \emph{rainbow} cycle. The necessary condition for a rainbow cycle to come from a factorization is simply that given a color of an edge of the cycle, no edge of that color touches any of the remaining vertices of the cycle, as that is precisely the condition for a zero to appear in the factorized matrix. A cycle with that property will be called a \emph{factorizing cycle}. We proved:

\begin{proposition}
For any $n$, $\rankh(S_n)$ is the smallest integer $k$ for which $J(k,\floor{\frac{k}{2}})$ has a length $n$ factorizing cycle.
\end{proposition}

In Figures \ref{fig:johnson_cor} and \ref{fig:johnson_ciclo} we can check the above coloring for $k=5$, as well as a factorizing cycle for it.

\begin{figure}[t]
  \centering
  \begin{minipage}[b]{0.4\textwidth}
    \includegraphics[width=\textwidth]{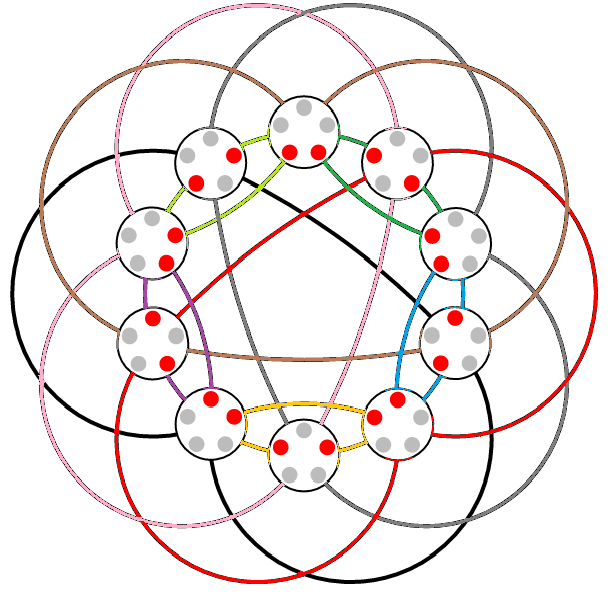}
    \caption{$J(5,2)$ with its coloring.}
    \label{fig:johnson_cor}
  \end{minipage}
  \hfill
  \begin{minipage}[b]{0.4\textwidth}
    \includegraphics[width=\textwidth]{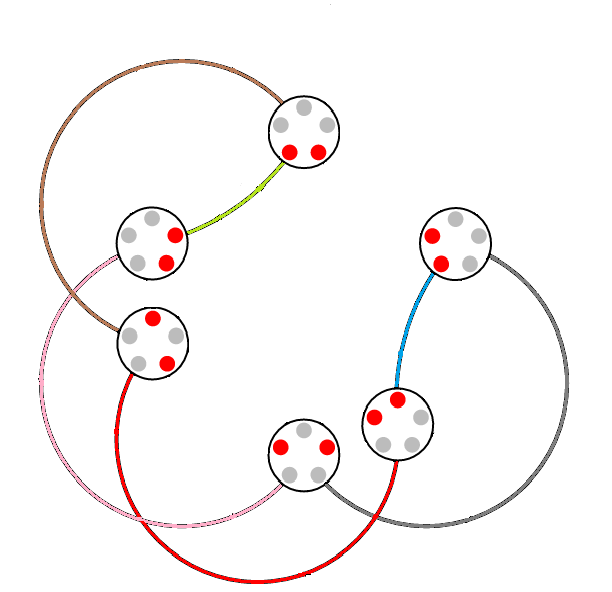}
    \caption{Factorizing cycle for $J(5,2)$.}
    \label{fig:johnson_ciclo}
  \end{minipage}
\end{figure}

\subsection{Numerical results}

Based on the graph interpretation above, we developed a depth search algorithm exploiting symmetries to verify the existence of factorizing cycles in a Johnson graph. The results can be found in Table~\ref{table:carboolhom}.
For some cases (marked with *), we managed to find factorizations of the given size, but could not rule out lower sized factorizations. In \cite{barefoot1986biclique} results are presented until $n=33$, but no details of the computation are presented, so it is unclear if smaller factorizations were ruled out or not. It is also suggested in that paper that the rank is likely $9$ for $n=34$ through $52$. Our results allow an update to that conjecture, suggesting that rank $9$ $n$-gons are likely those with $n=35-55$. The full results and algorithms used are available in \url{www.mat.uc.pt/~jgouveia/polyrank.html}.

\begin{table}[h!]
	\centering
	{\begin{tabular}{|c|c|c|c|c|c|c|c|c|c|c|c|c|c|c|c|c|c|c|c|c|c|c|c|c|c|c|}
	\hline
	$n$ & 3 & 4 & 5-6 & 7-9 & 10-21 & 22-34 & 35-40 & 41-55 & 56 - 78 & 79-91\\ \hline \hline
	$car_{bool}^{hom}(S_n)$ & 3 & 4 & 5 & 6 & 7 & 8 & 9$^*$ & 9 & 10$^*$ & 10\\
    \hline
	\end{tabular}}
	\caption{Homogeneous boolean rank values ($^*$ marks upper bounds not proven to be exact).}
	\label{table:carboolhom}
\end{table}

It is interesting to compare these values with the best known lower bounds. As mentioned before, most usual combinatorial lower bounds are not very effective for the polygon case. In \cite{vandaele2015linear},  Vandaele et al. propose an improvement for the bound $S_n$ specifically in the case of $n$-gons, the bound
$${S(n)}^+:= \min \left\{k: n\leq \frac{k-\floor{k/2}}{k-1}\binom{k}{\floor{k/2}}\right\}.$$
Comparing the values of $\rankh(S_n)$ and those of ${S(n)}^+$, we can see that they coincide for $3-6$, $8-9$, $13-21$,
$24-34$, $41-55$ and $79-91$ so in all these cases the factorizations found are optimal. In particular the true value of the homogeneous boolean rank was computed in those cases and it matches the usual boolean rank.

\section{Positive Semidefinite Ranks}\label{sec:psd}

In general, not very much is known about positive semidefinite ranks of polygons. For the real positive semidefinite rank, it is known that the rank of triangles and quadrilaterals is $3$ and everything else is at least $4$ (see \cite{Gouveia2013}). It is also known that pentagons and hexagons have always rank $4$ (see \cite{Gouveia2015}) and that the regular $8$-gon has semidefinite rank $4$ (see \cite{thesisvand} for an explicit factorization). Apart from that very little is known but numerical observations in \cite{thesisvand} led to the conjecture that $\rankpsd(S_n)=\ceil{\log_2(n)}+1$.

For complex semidefinite rank, nothing is really known, except what is given by the obvious inequality $\rankpsd(S_n) \geq \rp(S_n)$. In this section we make some inroads into this question and in particular we show some unexpected behaviour of this rank. We start by reformulating Proposition 3.2 and Theorem 3.5 of \cite{Gouveia2013} in the complex case. The proofs are omitted since they are virtually the same as those in that paper.

\begin{theorem}
Let $S_{P}$ be the slack matrix of a $d$-polytope $P$. Then $$ \rp S_{P}\ge d+1 .$$ Furthermore, when equality holds, every \psdc-factorization of size $d+1$ of $S_{P}$ uses only rank one matrices as factors.
\end{theorem}

The polytopes for which equality holds are said to be \psdc-minimal. The characterization of these objects can be done using the following theorem, which gives necessary and sufficient conditions for \psdc-minimality.

\begin{theorem}\label{root}
A $d$-polytope $P$ with slack matrix $S_{P}\in \RR^{f\times v}_{+}$ is \psdc-minimal if and only if there exists a matrix $M \in \CC^{f\times v}$ with $\rank M=d+1$ such that $$S_{P}= |M| \odot |M|.$$
\end{theorem}

Here, $\odot$ stands for the Hadamard entrywise product of matrices, while by $|M|$ we mean the matrix whose entries are the absolute values of the entries of $M$. We now recall a tool that was useful for establishing results related to the real psd-minimality of polytopes in the paper \cite{1506.00187}.

\begin{definition}
The \emph{symbolic slack matrix} of a $d$-polytope $P$ is the matrix $S_{P}(x)$, obtained by replacing all positive entries in the slack matrix $S_{P}$ of P with distinct variables $x_{1},\dots,x_{t}$.
\end{definition}

Note that the slack matrix of a polytope $P$ can be recovered by evaluating the symbolic one for some particular $\xi = (\xi_{1},\dots, \xi_{t}) \in \RR^{t}_{+}$. Also, the slack matrix of any other polytope $Q$ which is combinatorially equivalent to $P$ can, up to permutations of rows and columns, be obtained in this way for some $\xi \in \RR^{t}_{+}$.

\begin{example}\label{pent}
The slack matrix $S_{P}$ and the symbolic slack matrix $S_{P}(x)$ of the regular pentagon are:
$$
S_{P}=
\begin{pmatrix}
0 & 0 & 1 & \varphi &1 \\1 & 0 & 0 & 1 & \varphi \\ \varphi &1  & 0 & 0 & 1 \\1 & \varphi & 1 & 0 & 0 \\0 & 1 & \varphi & 1 & 0
\end{pmatrix} \qquad
S_{P}(x)=
\begin{pmatrix}
0 & 0 & x_{11} & x_6 & x_{1} \\
 x_2 & 0 & 0 & x_{12} & x_{7} \\
 x_8 & x_3 & 0 & 0 & x_{13}\\
x_{14} &  x_{9} & x_4 & 0 & 0\\
0 & x_{15} & x_{10} & x_{5} &0
\end{pmatrix},
$$
where the $\varphi$ is the golden ratio.
In this case, $S_{P}$ can be obtained as $S_{P}(\xi)$ with $\xi=(1,1,1,1,1,\varphi ,\varphi ,\varphi ,\varphi ,\varphi ,1,1,1,1,1)$.
\end{example}

\begin{remark} \label{scaling}
As the \psdc-minimality of a slack matrix $S_{P}$ is invariant under scalings of rows and columns, it is possible to make some of its entries be equal to one. This means we may assume that several of the variables in the corresponding slack matrix $S_{P}(x)$ have also been set to one. For instance, the symbolic slack matrix $S_{P}(x)$ of the pentagon in Example \ref{pent} can be reduced to $S_{P}'(x)$.
$$S_{P}'(x)=\begin{pmatrix}
0 & 0 & 1 & x_{1} & 1  \\
1 & 0 & 0 & 1& x_{2}  \\
x_{3} & 1 & 0 & 0 & 1  \\
1 & x_{4} & 1 & 0 & 0  \\
0 & 1 & x_{5} & x_{6} & 0
\end{pmatrix}$$
Likewise, any complex matrix $M$ such that $S_{P}= |M| \odot |M|$ and in the conditions of Theorem \ref{root} can be rescaled in the same way. In fact, this procedure does not change the rank of M nor the \psdc-rank of $S_{P}$. This happens because the changes in the absolute values of the entries of $M$ just correspond to rescalings of $S_{P}$ and the changes in their complex phase are not relevant. That means for any matrix $M$ in the conditions of Theorem \ref{root}, there will be a  version $M'$ scaled exactly like $S_{P}$.
This version can be obtained as $M'=S_{P}'(\zeta)$, for some $\zeta=(\zeta_{1},\dots,\zeta_{t}) \in \CC^{t}$.
Because of this, we will allow, in what follows, the symbolic slack matrix $S_{P}(x)$ of a polytope $P$ to be evaluated using a complex vector $\zeta\in\CC^{t}$.
\end{remark}

Unfortunately, the tools developed in \cite{1506.00187} to study the real case, fail to extend to the complex case in any meaningful way. We develop instead an alternate weaker obstruction that works in the complex case, and that will prove fundamental to attain our new results.

\begin{lemma}[Combined Trinomial Obstructions]\label{TO}
Consider a \psdc-minimal $d$-polytope $P$ with slack matrix $S_{P}$. If its symbolic slack matrix $S_{P}(x)$ has a $d+2$-minor that is of the form $x^{a}-x^{b}+x^{c} , a,b,c\in \NN^{t}$, then, for every $\zeta \in \CC^{t}$ such that $S_{P}=|S_{P}(\zeta)|\odot |S_{P}(\zeta)|$, with $\rank S_{P}(\zeta)=d+1$, we have $$ \Re\left( \frac{\zeta^{c}}{\zeta^{a}} \right)=0.$$
\end{lemma}

\begin{proof}
As $S_{P}(\zeta)$ and $S_{P}=S_{P}(\xi)$, for some $\xi \in \RR^{t}$, both { have rank $d+1$}, all their $d+2$ minors are identically zero. This means we have $\zeta^{a}-\zeta^{b}+\zeta^{c}=0$ and $\xi^{a}-\xi^{b}+\xi^{c}=|\zeta^{a}|^{2}-|\zeta^{b}|^{2}+|\zeta^{c}|^{2}=0$, because $S_{P}(\xi)=|S_{P}(\zeta)|\odot |S_{P}(\zeta)|$. The first equation can be put in the form, $\zeta^{b}=\zeta^{a}+\zeta^{c}$, and, substituting back in the second, we get $\conj{\zeta^{a}}\zeta^{c} + \zeta^{a}\conj{\zeta^{c}}=0$. Dividing this by $2\conj{\zeta^{a}}\zeta^{a}$, we get the result.
\end{proof}

\begin{corollary}
The pentagon is not \psdc-minimal.
\end{corollary}
\begin{proof}
Assume the pentagon with slack matrix $S_{P}$ is \psdc-minimal and consider its symbolic slack matrix $S_{P}(x)$ given in Remark \ref{scaling}. The 4-minors $m_{i,j}(x)$ of $S_{P}(x)$ obtained by deleting the $i$th row and the $j$th column, for $(i,j)=(5,5),(5,1), (3,4)$, are, respectively and up to sign,
\begin{align*}
m_{5,5}(x)&=x_{1}-x_{3}x_{4}+1\\
m_{5,1}(x)&=x_{4}-x_{1}x_{2}+1\\
m_{3,4}(x)&=x_{2}-x_{4}x_{5}+1.
\end{align*}

By Lemma \ref{TO}, the complex entries $\zeta_{1}, \zeta_{2}$ and $\zeta_{4}$ of any matrix $S_{P}(\zeta), \zeta \in \CC^{t},$ with $\rank S_{P}(\zeta)=d+1$ such that $S_{P}=|S_{P}(\zeta)|\odot |S_{P}(\zeta)|$, are pure imaginary.
 Also, all of the $S_{P}(\zeta)$ 4-minors are identically zero. In particular $m_{5,1}(\zeta)=\zeta_{4}-\zeta_{1}\zeta_{2}+1=0$, which is a contradiction. This denies the existence of any matrix $M$ in the conditions of Theorem \ref{root}.
\end{proof}

\begin{remark}
The proof above cannot be extended directly to $n$-gons with $n>5$. In fact, if one considers the slack matrix of the regular hexagon $S_{P}$, it is possible to find a matrix $M\in \CC^{6\times6}$ with $\rank M=3$ such that $S_{P}=|M|\odot|M|$. According to Theorem \ref{root}, this is equivalent to say that this polygon is \psdc-minimal. The matrices $S_{P}$ and $M$ are:
$$
S_{P}=
\begin{pmatrix}
0 & 0 & 1 & 2 & 2 & 1 \\
1 & 0 & 0 & 1 & 2 & 2 \\
2 & 1 & 0 & 0 & 1 & 2 \\
2 & 2 & 1 & 0 & 0 & 1 \\
1 & 2 & 2 & 1 & 0 & 0 \\
0 & 1 & 2 & 2 & 1 & 0
\end{pmatrix}
\quad
M=
\begin{pmatrix}
0&0&1&\sqrt{2}&\sqrt{2}&1\\
1&0&0&1&1-i&\sqrt{2}\\
1+i&1&0&0&1&\sqrt{2}i\\
\sqrt{2}i&\sqrt{2}i&-1&0&0&-1\\
1&1+i&\sqrt{2}i&1&0&0\\
0&1&\sqrt{2}&1-i&1&0\\
\end{pmatrix}.
$$
In fact, this is contrary to what happens in the real case where the \psd-minimal polygons are just triangles and quadrilaterals.
\end{remark}

We have shown that $\rp(S_3)=\rp(S_4)=\rp(S_6)=3$ while also showing $\rp(S_5)\ge 4$. Since $4=\rankpsd(S_5)\ge \rp(S_5)$ we actually know that $\rp(S_5)=4$. This shows a very interesting property of $\rp(S_n)$: it is not an increasing sequence. This is very interesting in that while the sequence of ranks of regular $n$-gons is widely thought to be increasing for the nonnegative, boolean and real semidefinite ranks, that fact was not ever proved for any of them, and we just proved it false in the complex semidefinite case. This gives evidence that one should be very careful when assuming, even implicitly, such type of behaviour in numerical experiments.

\bibliography{bibliografia}
\bibliographystyle{alpha}

\end{document}